\documentclass[12pt]{amsart}

\usepackage{amssymb}
\usepackage{amsthm}

\usepackage[usenames]{color}

\usepackage[margin=1in]{geometry}
\usepackage[foot]{amsaddr}

\theoremstyle{plain}
\newtheorem{proposition}{Proposition}[section]
\newtheorem{theorem}[proposition]{Theorem}
\newtheorem{lemma}[proposition]{Lemma}

\theoremstyle{definition}

\newtheorem{definition}[proposition]{Definition}

\theoremstyle{remark}

\DeclareMathOperator{\tr}{tr}

\DeclareMathOperator{\RR}{\mathbb{R}}

\DeclareMathOperator{\JJ}{\mathbb{J}}

\newcommand{\abs}[1]{\left|#1\right|}

\newcommand{\norm}[1]{\left\|#1\right\|}
\newcommand{\wt}[1]{\widetilde{#1}}

\begin{document}

\title{Asymptotically harmonic manifolds without focal points}
\author{Andrew M. Zimmer}\address{Department of Mathematics, University of Michigan, Ann Arbor, MI 48109.}
\email{aazimmer@umich.edu}
\date{\today}

\begin{abstract}
In this note we show that a compact asymptotically harmonic manifold without focal points is either flat or a rank one locally symmetric space. 
\end{abstract}

\maketitle

\section{Introduction}

A complete Riemannian manifold $(M,g)$ without conjugate points is called \emph{asymptotically harmonic} if the mean curvature of the horospheres is constant. Examples of asymptotically harmonic manifolds include flat spaces and rank one locally symmetric spaces. The purpose of this note is to show that these are the only compact asymptotically harmonic manifolds without focal points.

\begin{theorem}
\label{thm:main}
If $(M,g)$ is a compact asymptotically harmonic manifold without focal points, then $(M,g)$ is either flat or a rank one locally symmetric space. 
\end{theorem}

If $(M,g)$ is also assumed to have negative curvature, then $(M,g)$ is a rank one locally symmetric space by the combined results of Foulon and Labourie~\cite{FL92}, Benoist, Foulon, and Labourie~\cite{BFL92}, and Besson, Courtois, and Gallot~\cite{BCG95}. As observed by Knieper~\cite[Theorem 3.6]{knieper11}, these arguments actually show that any compact asymptotically harmonic manifold whose geodesic flow is Anosov is a rank one locally symmetric space. Using their work, Theorem~\ref{thm:main} follows from:

\begin{theorem}
\label{thm:main2}
If $(M,g)$ is a finite volume asymptotically harmonic manifold without focal points, then $(M,g)$ is either flat or the geodesic flow on $SM$ is Anosov with respect to the Sasaki metric. 
\end{theorem}

If $(M,g)$ is not assumed to be compact there exist nonsymmetric homogeneous Hadamard manifolds, namely the Damek-Ricci spaces, which are asymptotically harmonic ~\cite{DR92}. Asymptotically harmonic manifolds without conjugate points have been classified in dimension 3~\cite{HKS07,knieper94,SS08,shah11} and in the Einstein, homogeneous case (see Heber~\cite{heber06}). 

Asymptotically harmonic manifolds are related to harmonic manifolds. A complete Riemannian manifold $(M,g)$ is called \emph{harmonic} if about any point the geodesic spheres of sufficiently small radii are of constant mean curvature. In the introduction of his 1990 paper on harmonic manifolds Szabo~\cite{szabo90} observes that a harmonic manifold without conjugate points is also ``globally'' harmonic: about any point the geodesic spheres of any radii are of constant mean curvature. Thus every harmonic manifold without conjugate points is an asymptotically harmonic manifold. There has been a great deal of progress towards classifying all harmonic manifolds and we refer the reader to the introduction in Heber~\cite{heber06} for a survey of the literature.

Recently Knieper~\cite{knieper11} demonstrated Theorem~\ref{thm:main2} under the assumption that $(M,g)$ is a harmonic manifold without focal points. Unlike the argument presented in this paper, he does not require that $(M,g)$ has finite volume. The methods used here, while of a similar flavor to Knieper's argument, differ substantially. One key feature of harmonic manifolds that Knieper exploits is that the volume density in normal coordinates depends only on the radius. This symmetry is unavailable to us in the asymptotically harmonic case, instead we use a recent generalization of the rank rigidity theorem to manifolds without focal points by Watkins~\cite{watkins11} to reduce to the rank one case. We then show that a function related to curvatures of both the stable and unstable horospheres is constant and nonzero. This implies, by a theorem of Bolton~\cite{bolton79}, that the geodesic flow is Anosov.

\section{Preliminaries}

Every Riemannian manifold $(M,g)$ considered here will be complete, $SM$ will denote the unit tangent bundle, and $\phi_t$ the geodesic flow on $SM$.

\subsection{Asymptotically harmonic manifolds:} We begin by introducing the stable and unstable Riccati solutions. The following discussion closely follows the work of Green~\cite{green58}. Assume $M$ is a complete manifold without conjugate points and $v \in SM$. If $\gamma_v$ is the unique unit speed geodesic with $\gamma^\prime_v(0)=v$ and $R(t)=R(\gamma^\prime(t),\cdot)\gamma^\prime(t)$ is the curvature tensor along $\gamma$, let $\JJ(t):T_{\gamma(t)}M \rightarrow T_{\gamma(t)}M$ be the unique solution to the differential equation 
\begin{align*}
\JJ^{\prime\prime}(t)+R(t)\JJ(t)=0
\end{align*}
with $\JJ(0)=0$ and $\JJ^\prime(0)=Id$. As $M$ has no conjugate points, the endomorphism $\JJ(t)$ is invertible for all $t \neq 0$. Next for $T \in \RR$ define
\begin{align*}
E_T(t)=\JJ(t)\int_{t}^T \JJ^{-1}(s)\left(\JJ^{-1}(s)\right)^* ds
\end{align*}
where $*$ is the adjoint operator with respect to the inner product $g$. As $\JJ^\prime(t)\JJ^{-1}(t)$ is symmetric, the tensors $E_T$ satisfy the Jacobi equation
\begin{align}
\label{eq:stable_jacobi_tensor}
E_T^{\prime\prime}(t)+R(t)E_T(t)=0
\end{align}
and $E_T(0)=Id$ and $E_T(T)=0$. Next define $U_T^s(v):=E_T^\prime(0)$ and let $\phi_t:SM \rightarrow SM$ be the geodesic flow on $SM$. Using equation~\ref{eq:stable_jacobi_tensor} the paths $t \rightarrow U_T^s(\phi_tv)$ solves the differential equation
\begin{align*}
(U_T^s)^\prime + (U_T^s)^2+R = 0.
\end{align*}
Using the defining equation for $E_T$ we also see that
\begin{align}
\label{eq:riccati_mono}
U_{T_2}^s(v) - U_{T_1}^s(v) = \int_{T_1}^{T_2} \JJ^{-1}(s)\left(\JJ^{-1}(s)\right)^* ds
\end{align}
for $T_2>T_1$ implying that $U_T(v)$ is monotonically increasing in $T$. Further the sequence $\{ U_T(v)\}_{T >0}$ can be shown to be bounded. If one assumes a lower bound on the curvature tensor, this follows from standard comparision theorems. More generally Green~\cite{green58} provides an argument showing that $\{U_T^s(v)\}_{T >0}$ is bounded whenever $R(t)$ is continuous. 

The above discussion implies that the endomorphisms $U_T^s(v):v^{\bot} \rightarrow v^{\bot}$ converges montonically to a endomorphism $U^s(v):v^{\bot} \rightarrow v^{\bot}$ as $T \rightarrow \infty$. By construction the path $t \rightarrow U^s(\phi_t v)$ also solves the differential equation
\begin{align*}
(U^s)^\prime + (U^s)^2+R =0.
\end{align*}

Also for each $v \in SM$, define the linear map $U^u(v):v^{\bot} \rightarrow v^{\bot}$ by setting $U^u(v) := -U^s(-v)$. Then by definition the path $t \rightarrow U^u(\phi_t v)$ also solves the differential equation
\begin{align*}
(U^u)^\prime + (U^u)^2+R =0.
\end{align*}

The tensors $U^s$ and $U^u$ are called the \emph{stable and unstable Riccati solutions.} When $(M,g)$ is a negatively curved, compact Riemannian manifold the tensor $-U^s$ is the second fundamental form of the horospheres and determine the stable and unstable foliations (see Eberlein~\cite{eberlein73I}). Using the geometric interpretation we obtain the following definition for asymptotically harmonic manifolds.

\begin{definition} A complete Riemannian manifold $(M,g)$ without conjugate points is said to be \emph{asymptotically harmonic} if $\tr U^s(v)$ is constant on $SM$.
\end{definition}

When $(M,g)$ is asymptotically harmonic and compact the constant $-\tr U^s$ is equal to the volume growth entropy. To be more precise, let $(M,g)$ be a compact Riemannian manifold and $\wt{M}$ the universal cover of $M$. The \emph{volume growth entropy} of $M$ is defined to be:
\begin{align*}
h_{vol}(M) = \lim_{r\rightarrow \infty} \frac{ \log \text{Vol}( B_r(p))}{r}
\end{align*}
where $p \in \wt{M}$ and $B_r(p)$ is the ball of radius $r$ around $p$. Manning~\cite{manning79} showed that this limit always exists and is independent of $p$. Further Manning showed that when $M$ is nonpositively curved $h_{top}(M)=h_{vol}(M)$, where $h_{top}(M)$ is the topological entropy of the geodesic flow on $SM$. Freire and Ma\~{n}\'{e}~\cite{FM82} generalized this last result and showed that $h_{top}(M)=h_{vol}(M)$ when $M$ has no conjugate points. 

When $(M,g)$ is a compact asymptotically harmonic manifold and $\tr U^u(v) = -\tr U^s(v) \equiv \alpha$, then the constant $\alpha$ is actually equal to $h_{vol}(M)$ (see for instance Heber~\cite[Remark 2.2]{heber06}).

\subsection{Anosov geodesic flows} In this section we state a criterium for the geodesic flow being Anosov. Recall that for $v \in TM$ and $p=\pi(v)$ we have a canonical splitting 
\begin{align*}
T_v TM = T_p M \oplus T_pM
\end{align*}
where the first factor is the \emph{horizontal distribution} and the second is the \emph{vertical distribution}. Further if $v \in SM$, then 
\begin{align*}
T_v SM = \{ (X,Y) : Y \in v^{\bot}\}.
\end{align*}
Next define the \emph{stable and unstable Green subbundles} as
\begin{align*}
E^s(v) &= \{ (X,U^s(v)X) : X \in v^{\bot} \}, \\
E^u(v) &= \{ (X,U^u(v)X) : X \in v^{\bot} \}.
\end{align*}
We then have the following.

\begin{theorem}
\label{thm:anosov}
Let $(M,g)$ be a Riemannian manifold without conjugate points. If there exists $K,\delta >0$ such that $\norm{U^s(v)} \leq K$ for all $v \in SM$ and 
\begin{align*}
g(U^u(v)X-U^s(v)X,X) \geq \delta g(X,X) \text{ for all $v \in SM$ and $X \in v^{\bot}$}
\end{align*}
then the geodesic flow on $SM$ is Anosov with respect to the Sasaki metric.
\end{theorem}

When $A$ is a linear operator between normed spaces, $\norm{A}$ is the usual operator norm.

Theorem~\ref{thm:anosov} follows from a result of Bolton~\cite[Proposition 1]{bolton79}. He proves the above theorem with the additional assumption that the sectional curvature is bounded from below by $-K^2$ with $K \geq 0$. Given such a bound on the sectional curvature, one can use comparision theorems to deduce that $\norm{U^s(v)} \leq K$. In fact, this is the only way Bolton's argument uses the curvature bound and the above Theorem follows from his argument verbatim.

Eberlein proved the above theorem in the compact case~\cite{eberlein73I}. We also remark that when the geodesic flow is Anosov, then the stable and unstable Green subbundles are the stable and unstable distributions for the Anosov flow.

\section{Reduction to the rank one case}

Using a recent generalization of the rank rigidity theorem we reduce to the case in which $(M,g)$ has geometric rank one. If $(M,g)$ is a Riemannian manifold and $v \in SM$ then the \emph{rank} of $v$ is the dimension of the vector space of bounded Jacobi fields along the geodesic $\gamma_v(t)$. Then define the \emph{rank} of $M$ to be minimum rank over all $v \in SM$. 

A vector $v \in SM$ is said to be \emph{$\text{Isom}(M)$-recurrent} if there are sequences $(t_n) \in \RR$ and $\psi_n \in \text{Isom}(M)$ such that $\psi_n \circ \phi_{t_n}(v) \rightarrow v$ as $n \rightarrow \infty$. A recent generalization of the rank rigidity theorem due to Watkins applies to manifolds without focal points.

\begin{theorem}\cite{watkins11}
Suppose $(\wt{M},g)$ is a simply connected Riemannian manifold without focal points. If a dense subset of $S\wt{M}$ is $\text{Isom}(\wt{M})$-recurrent and $\wt{M}$ has rank greater or equal to 2, then $\wt{M}$ is either a Riemannian product or a symmetric space.
\end{theorem}

In the case when $(\wt{M},g)$ has nonpositive curvature, Theorem 3.1 was originally proved by Ballmann~\cite{ballmann85} and Burns and Spatzier~\cite{BS87}.

When $(\wt{M},g)$ has a finite volume quotient, then a dense subset of $S\wt{M}$ will be $\text{Isom}(\wt{M})$-recurrent by the Poincar\'{e} recurrence theorem.

Notice that a manifold is asympototically harmonic if and only if its universal cover is asymptotically harmonic. We now show that all higher rank asymptotically harmonic manifolds are flat. Ledger~\cite{ledger57} showed that irreducible symmetric harmonic manifolds must have rank one (see also Eschenburg~\cite{eschenburg80}). The same is true for asymptotically harmonic manifolds, for instance if $(M,g)$ is a non-compact symmetric space the discussion in~\cite{eschenburg80} implies that
\begin{align*}
\tr U^s(v) = -\sum_{\alpha \in Rt} k_{\alpha} \abs{\alpha(v)}
\end{align*}
where $Rt$ is the set of roots of $M$ and $k_{\alpha}$ are nonnegative integers. If rank is greater or equal to two and $(M,g)$ is irreducible, the right hand side will not be constant.

\begin{lemma} 
\label{lem:no_higher_rank}
Suppose $(\wt{M},g)$ is an irreducible symmetric space. Then $(\wt{M},g)$ is asymptotically harmonic if and only if $(\wt{M},g)$ has rank one.
\end{lemma} 

It is a result of Lichnerowicz~\cite{lichnerowicz44} that a harmonic manifold which is a Riemannian product must be flat. The following two lemmas show that the same is true for asymptotically harmonic manifolds without focal points.

\begin{lemma}
\label{lem:prod1}
Supppose $(M,g)=(M_1 \times M_2,g_1 \times g_2)$ is a Riemannian product and $(M,g)$, $(M_1,g_1)$, $(M_2,g_2)$ have no conjugate points. Let $U^s, U_1^s, U_2^s$ be the stable Riccati solutions for $M,M_1,M_2$ respectively. If $M$ is asymptotically harmonic then $M_1$ and $M_2$ are asymptotically harmonic and $\tr U^s = \tr U_1^s = \tr U_2^s =0$.
\end{lemma}

\begin{proof} 
Suppose $M$ is asymptotically harmonic with $\tr U^s \equiv \alpha$. Viewing $T_{(x_1,x_2)}M = T_{x_1} M_1 \times T_{x_2} M_2$ and picking $(v_1,v_2) \in SM$ we have
\begin{align}
\label{eq:prod_lemma}
\tr U^s(v_1,v_2) = \norm{v_1} \tr U_1^s(v_1/\norm{v_1}) + \norm{v_2} \tr U_2^s(v_2/\norm{v_2}).
\end{align}
To see this let $\nabla,\nabla^1,\nabla^2$ be the Levi-Civita connections on $M,M_1,M_2$ respectively. Then
\begin{align*}
\nabla_{(Y_1,Y_2)} (X_1,X_2) = (\nabla^1_{Y_1} X_1, \nabla^2_{Y_2} X_2)
\end{align*}
implying, using the notation in Section 2, that
\begin{align*}
E_T(t) = \left(E^1_T\left(\norm{v_1}t\right), E^2_T\left(\norm{v_2}t\right)\right)
\end{align*}
where $E_T,E^1_T,E^2_T$ are the Jacobi tensors along the geodesics $\gamma_{(v_1,v_2)}$ in $M$, $\gamma_{v_1/\norm{v_1}}$ in $M_1$, $\gamma_{v_2/\norm{v_2}}$ in $M_2$ respectively that are equal to the identity at $t=0$ and vanish at $t=T$. Then taking the limit as $T\rightarrow \infty$ of $E^\prime_T(0)$ yields equation~\ref{eq:prod_lemma}.

So if $v_1 \in SM_1$ and $v_2 \in SM_2$, plugging in $(v_1,0)$ and $(0,v_2)$ into Equation~\ref{eq:prod_lemma} yields $\tr U_1^s(v_1) = \alpha$ and $\tr U_2^s(v_2) \equiv \alpha$. But plugging in $(v_1/\sqrt{2},v_2/\sqrt{2})$ in Equation~\ref{eq:prod_lemma} yields 
\begin{align*}
\alpha = \tr U^2(v_1/\sqrt{2},v_1/\sqrt{2}) = \sqrt{2} \alpha
\end{align*}
and so $\alpha =0$.
\end{proof}

\begin{lemma}
\label{lem:prod2}
Let $(M,g)$ be an asymptotically harmonic manifold without focal points and $\tr U^s \equiv 0$, then $(M,g)$ is flat.
\end{lemma}

\begin{proof}
As $M$ has no focal points $U^s$ is negative semidefinite (see for instance see the proof of Corollary 3.3 in Eberlein~\cite{eberlein73I}), further $\tr U^s  \equiv 0$ thus as $U^s$ is symmetric, $U^s \equiv 0$ and then the Riccati equation 
\begin{align*}
(U^s)^\prime + (U^s)^2+R=0
\end{align*}
implies that $R \equiv 0$.
\end{proof}

\section{A useful function}

As in~\cite{HKS07,knieper11, SS08} we consider the map 
\begin{align*}
v \in SM \rightarrow V(v)=U^u(v)-U^s(v) \in \text{End}(v^{\bot}).
\end{align*}
Let $\phi_t:SM \rightarrow SM$ be the geodesic flow on $SM$. Because $U^{s}, U^u$ solve the Riccati equation, for any $v \in SM$ the path $t \rightarrow V(\phi_t v)$ satisfies the differential equation
\begin{align*}
V^\prime = XV + VX
\end{align*}
where $X(t) = -\frac{1}{2}(U^u(\phi_t v)+U^s(\phi_t v))$. We now relate $V(v)$ to the rank of $v \in SM$ and to Bolton's criterium for the geodesic flow being Anosov. We will need the following description of stable and unstable Jacobi fields.

\begin{lemma}
\label{lem:kerV}
Let $(M,g)$ be a Riemannian manifold without focal points. Suppose that $\gamma_v$ is the unit speed geodesic and $J(t)$ is a Jacobi field along $\gamma_v$. Then $\norm{J(t)} \leq C$ for all $t \in \RR$ if and only if $(J(0),J^\prime(0)) \in E^s(v) \cap E^u(v)$. In particular, $\det V(v) = 0$ if and only if $v$ has rank two or more.
\end{lemma}

Lemma~\ref{lem:kerV} is well known in the case in which $(M,g)$ has no conjugate points and sectional curvature bounded below (see for instance~\cite[Corollary 2.14]{eberlein73I}). Unable to find a reference for the situation above, we include the short proof. 

Recall that if $J$ is a Jacobi field in a manifold without focal points and $J(0)=0$ then 
\begin{align}
\label{eq:jacobi_no_fp}
\frac{d}{dt} \norm{J(t)}^2 > 0.
\end{align}

\begin{proof}
Suppose $(J(0),J^\prime(0)) \in E^s(v) \cap E^u(v)$, then $J$ is the limit of Jacobi fields $J_T$ with $J_T(0)=J(0)$ and $J_T(T)=0$ as $T \rightarrow \infty$. Then applying equation~\ref{eq:jacobi_no_fp} to the Jacobi field $Y(t)=J_T(T-t)$ we see that 
\begin{align*}
\norm{J_T(t)} \leq \norm{J(0)}
\end{align*}
for $0 \leq t \leq T$. Further for $t$ fixed, $J_T(t) \rightarrow J(t)$ so we obtain that $\norm{J(t)} \leq \norm{J(0)}$ for all $t \geq 0$. Now the Jacobi field $J$ is also the limit of Jacobi fields $J_T$ with $J_T(0)=J(0)$ and $J_T(T)=0$ as $T \rightarrow -\infty$. Thus the above argument shows that $\norm{J(t)} \leq \norm{J(0)}$ for all $t \in \RR$.

Now suppose $J$ is a Jacobi field along $\gamma_v$ such that $\norm{J(t)} \leq C$. Let $Y$ be the stable Jacobi field with $Y(0)=J(0)$ and $Y^\prime(0)=U^s(v)J(0)$. Then $J-Y$ is a Jacobi field with $(J-Y)(0)=0$ and by the above argument $\norm{(J-Y)(t)} \leq C+\norm{Y(0)}$ for $t \geq 0$. If $Y^\prime(0)=J^\prime(0)$ then $(J(0),J^\prime(0)) \in E^s(v)$. Otherwise $(J-Y)^\prime(0)\neq 0$ and by a result of M.S. Goto~\cite[Theorem 1]{goto78},
\begin{align*}
\norm{(J-Y)(t)} \rightarrow \infty
\end{align*}
as $t \rightarrow \infty$ which is impossible. A similar argument shows that $(J(0),J^\prime(0)) \in E^u(v)$
\end{proof}

We will next show that $v \rightarrow \det V(v)$ is continuous and invariant under the geodesic flow. The following lemma is given in~\cite{HKS07,SS08}, but for completeness we include the short proof.

\begin{lemma}
\label{lem:inv}
Let $(M,g)$ be an asymptotically harmonic manifold. Then the map $v \rightarrow \det(V(v))$ is invariant under the geodesic flow.
\end{lemma}

\begin{proof}
For $v \in SM$ we will show that the function $t \rightarrow \det(V(\phi_t v))$ is constant. If $\det(V(v))=0$, then the rank of $v$ is 2 or more which implies that $\det(V(\phi_t v))=0$ for all $t$.

If $\det(V(v)) \neq 0$, then letting $V(t) = V(\phi_t v)$ we obtain
\begin{align*}
\frac{d}{dt} \log \det V = \tr \dot{V} V^{-1}=\tr(XV+VX)V^{-1} =2\tr X = -\tr U^u - \tr U^s =0
\end{align*}
as $U^u(v)=-U^s(-v)$ and $\tr U^s \equiv \alpha$.
\end{proof}

In general the map $v \rightarrow U^s(v)$ may only be measurable, but in the case when $M$ is asymptotically harmonic it is not hard to show continuity.

\begin{lemma} If $(M,g)$ is asymptotically harmonic, the map $v \rightarrow U^s(v)$ is continuous.
\end{lemma}

\begin{proof}
It is enough to verify the lemma on the universal cover of $M$, so we may as well assume that $M$ is simply connected. 

The map $ v \rightarrow \tr U^s(v)$, being constant, is continuous. Further by equation~\ref{eq:riccati_mono}, 
\begin{align}
\label{eq:mono_conv}
g(U_{T_1}^s(v) X, X) \leq g(U_{T_2}^s(v)X,X) \leq g(U^s(v)X,X)
\end{align}
for $T_1 < T_2$ and $X \in v^{\bot}$. So $\tr U_T^s(v) \rightarrow \tr U^s(v)$ monotonically and by Dini's Theorem $\tr U_T^s(v)$ converges to $\tr U^s(v)$ locally uniformally as $T \rightarrow \infty$. 

We now claim that the endomorphisms $U_T^s(v)$ converges to $U^s(v)$ locally uniformally. Let $\lambda_T(v)$ be the maximum eigenvalue of $U^s(v)-U_T^s(v)$. As $U^s(v)-U_T^s(v)$ is positive semidefinite we have
\begin{align*}
0 \leq U^s(v)-U_T^s(v) \leq \lambda_T(v) Id \leq \tr\big(U^s(v)-U_T^s(v)\big) Id
\end{align*}
and as the right hand side of the above equation coverges locally uniformally to zero, so does $U^s(v)-U_T^s(v)$.

We will now show that the map $v \rightarrow U_T^s(v)$ is continuous, which will prove the Lemma. Consider the function
\begin{align*}
b_{v,T}(x) = d(\gamma_v(T),x)-T
\end{align*}
as $M$ has no conjugate points and is simply connected, $b_{v,T}$ is $C^\infty$ on $M \setminus \{ \gamma_v(T)\}$. Further, $U^s_T(v)$ is equal to $-\text{Hess}_{\pi(v)} b_{v,T}$ and
\begin{align*}
v \rightarrow -\text{Hess}_{\pi(v)} b_{v,T}
\end{align*}
is continuous in $v$ when $T \neq 0$. 
\end{proof}

\section{Proof of the main theorems}

Assume $(M,g)$ is a finite volume, asymptotically harmonic manifold without focal points. If $\wt{M}$ is the universal cover of $M$, the rank rigidity theorem says that $\wt{M}$ is either a Riemannian product, a higher rank symmetric space, or rank one. By Lemma~\ref{lem:no_higher_rank}, Lemma~\ref{lem:prod1}, and Lemma~\ref{lem:prod2}, if $\wt{M}$ has higher rank then $(M,g)$ is flat. 

It remains to consider the case in which $(M,g)$ has rank one. Ballmann~\cite{ballmann82} showed that the geodesic flow is topologically transitive on $SM$ if $M$ is a rank one finite volume manifold with nonpositive curvature. Hurley~\cite{Hurley86} showed the same is true if nonpositive curvature is replaced by no focal points. Then as $v \rightarrow \det V(v)$ is continuous and invariant under the geodesic flow, it must be constant. Further $\det V(v)$ is not identically zero because $M$ has rank one. Thus $\det V(v) \equiv \beta \neq 0$. 

Now by hypothesis $\tr V(v) = \tr U^u(v)-\tr U^s(v)=2\alpha$ is constant and as $(M,g)$ has no focal points $V(v)=U^u(v)-U^s(v)$ is positive semidefinite (again see the proof of Corollary 3.3 in Eberlein~\cite{eberlein73I}). In particular, the largest eigenvalue of $V(v)$ is no larger than $2\alpha$, implying that the smallest eigenvalue of $V(v)$ is at least $\beta (2\alpha)^{1-n}$. As $V(v)$ is symmetric, this implies that
\begin{align*}
g(V(v)X,X) \geq \beta (2\alpha)^{1-n} g(X,X) \text{ for all $X \in v^{\bot}$}.
\end{align*}
As $\tr U^s(u) = -\alpha$ and $U^s(u)$ is negative semidefinite, we also obtain that $\norm{U^s(v)} \leq \alpha$. Now Theorem~\ref{thm:anosov} implies that the geodesic flow on $SM$ is Anosov. This proves Theorem~\ref{thm:main2}.

Theorem~\ref{thm:main} now follows from the results of Foulon and Labourie~\cite{FL92}, Benoist, Foulon, and Labourie~\cite{BFL92}, and Besson, Courtois, and Gallot~\cite{BCG95}. See Knieper~\cite[Theorem 3.6]{knieper11} for details.

\subsection*{Acknowledgements}

I would like to thank Ralf Spatzier for many helpful conversations, Jordan Watkins for informing me of his recent generalization of the rank rigidity theorem, and Gerhard Knieper for several helpful comments on an earlier version of this paper.

\bibliographystyle{plain}
\bibliography{geom}

\end{document}